\documentclass[11pt]{article}

\usepackage{graphicx}
\usepackage{amssymb,epsfig}
\usepackage{amsmath}
\usepackage{color}
\usepackage[francais,english]{babel}
\usepackage[latin1]{inputenc}
\usepackage[T1]{fontenc}

\def\div{{\rm div }}

\def\R{\mathbb{R}}
\def\T{\mathbb{T}}

\def\E{\mathbb{E}}                    

\newtheorem{theorem}{Theorem}
\newtheorem{proposition}{Proposition}

\newtheorem{definition}{Definition}

\title{Enhanced sampling of multidimensional
       free-energy landscapes using adaptive 
       biasing forces}
        
\author{Christophe Chipot\footnote{On leave from Équipe de dynamique des 
                                   assemblages membranaires, UMR 7565, Nancy 
                                   Université, BP 239, 54506 
                                   Vand\oe uvre-lès-Nancy cedex, France.
                                   chipot@ks.uiuc.edu} 
        and 
        Tony Lelièvre\footnote{École des Ponts ParisTech, 6 et 8 avenue Blaise Pascal, 
                               77455 Marne-la-Vallée, France.    
                               lelievre@cermics.enpc.fr}        
        \\[0.2cm]        
 $^\ast$\emph{Theoretical and Computational Biophysics Group, }             \\
        \emph{Beckman Institute for Advanced Science and Engineering, }     \\
        \emph{University of Illinois, Urbana-Champaign. }                   \\ 
 $^\dag$\emph{Université Paris-Est, Cermics. }}

\begin{document}

\selectlanguage{english}

\maketitle

\abstract{We propose an adaptive biasing algorithm aimed at enhancing the 
sampling of multimodal measures by Langevin dynamics. The underlying idea 
consists in generalizing the standard adaptive biasing force method commonly used 
in conjunction with molecular dynamics to handle in a more effective fashion 
multidimensional reaction coordinates. The proposed approach is anticipated to be 
particularly useful for reaction coordinates, the components of which are weakly 
coupled, as illuminated in a mathematical analysis of the long-time convergence 
of the algorithm. The strength as well as the intrinsic limitation of the method 
are discussed and illustrated in two realistic test cases.}



\medskip

\section{Introduction}

Sampling of multimodal measures is a central problem in many 
scientific areas, such as statistical 
simulations, in particular molecular dynamics, which constitutes 
the primary focus of the present work. One standard approach to deal with such a 
situation consists in resorting to biasing techniques --- e.g.
importance sampling methods, in order to reduce the multimodal nature
of the targeted measure. Under these premises, the main difficulty is 
evidently to devise the correct bias.

One class of methods proposed in the framework of molecular 
dynamics and which has proven to be useful also for a variety of 
applications~\cite{chopin-lelievre-stoltz-10} are adaptive 
biasing numerical schemes. The underlying idea
here consists in designing adaptively the bias such that 
the new targeted measure be uniform 
along {\em a priori} chosen directions. Only these directions have to be chosen, 
but not the precise analytical expression of the bias. Examples of such 
class of approaches include the Wang-Landau algorithm~\cite{wang-landau-01}, 
non-equilibrium metadynamics~\cite{laio-parinello-02,bussi-laio-parinello-06} 
and the adaptive biasing 
force (ABF) method~\cite{darve-pohorille-01,henin-chipot-04}, which will 
constitute the main thrust of this contribution. 
The reader is referred to~\cite[Chapter 5]{lelievre-rousset-stoltz-10} for a 
general, mathematically-oriented presentation of adaptive 
methods.

Let us introduce the ABF method. In what 
follows, the algorithms will be presented in the framework of 
sampling of configurational space and overdamped Langevin 
dynamics, but generalization to sampling of phase space and standard Langevin 
dynamics is straightforward, as can be seen 
in Section~\ref{sec:num}. The canonical, Boltzmann-Gibbs, measure 
will be considered here:
\begin{equation}\label{eq:mu}
d\mu(q)=Z^{-1}\exp(-\beta V(q)) \, dq,
\end{equation}
where $\beta$ is proportional to the inverse temperature, $q  \in  {\mathcal D}$,  
$V:{\mathcal D}  \to \R$ is the so-called potential energy 
function, which is assumed to be a smooth function in the following,  $Z=\int_
{{\mathcal D}} \exp(-\beta  V(q))  \,  dq$ and ${\mathcal D}=\{ q,\, V(q) < 
\infty \}$ is the configurational space. The overdamped Langevin 
dynamics writes:
\begin{equation}
\label{eq:Q}
dQ_t=- \nabla V(Q_t) \, dt + \sqrt{2 \beta^{-1}} dB_t,
\end{equation}
where $B_t$ is an $n$-dimensional standard Brownian motion. 
Under mild conditions on~$V$, this dynamics is ergodic with respect to the measure $\mu$, i.e. trajectory or time averages converge to canonical 
averages.

For a multimodal measure $\mu$, the sampling obtained with $Q_t$ is, 
however, rather poor. Indeed, the typical problem with dynamics~\eqref{eq:Q} 
is that the process $Q_t$ remains trapped for long times in some 
metastable states. The purpose of an adaptive 
biasing force is to enhance sampling by 
subtracting from $V$ a potential such that the aforementioned
metastable features are eliminated. This relies on an assumed {\em a priori} 
knowledge of those ``coordinates'' that remain trapped 
--- viz. ``slow variables'' of the dynamics, also called ``collective 
variables'' or ``reaction coordinates''. This method can thus be seen as an 
adaptive importance sampling method.

In the following, emphasis will be put on the case
where two slow variables have been identified, namely $\xi_1:{\mathcal D}  
\to \T$, and $\xi_2:{\mathcal D}  \to \T$, where $\T$ denotes the
one-dimensional torus. Throughout the present work, $\xi_1$ and 
$\xi_2$ will be assumed to be smooth functions such that $|\nabla 
\xi_1|\neq 0$ and $|\nabla \xi_2|\neq 0$.

For simplicity, let us consider for the moment simple reaction coordinates: 
\begin{equation}\label{eq:simp_case}
{\mathcal D}=\T^n, \, \xi_1(x)=x_1, \,\xi_2(x)=x_2, 
\end{equation}
where the components of $x$ are referred to as: $x=(x_1,x_2,x_{3 
\ldots n})$. The case of general $\xi_i$'s will be discussed below, in Sections~
\ref{sec:GABF} and~\ref{sec:num}. In this simple framework, one standard 
ABF--like method is~\cite{lelievre-rousset-stoltz-07-b}:
\begin{equation} \label{eq:ABF_standard}
\left\lbrace
\begin{aligned}
&dY_t= \Big( - \nabla V + \sum_{\alpha=1}^2 \Gamma^\alpha_t \circ (\xi_1,\xi_2) 
\, 
\nabla \xi_\alpha    \Big)(Y_t) \, dt + \sqrt{2 \beta^{-1}} dB_t,\\
&\text{for $\alpha=1,2$}, \ \Gamma^\alpha_t (x_1,x_2) = \E(\partial_{x_\alpha} V
(Y_t) \, |\, (\xi_1,\xi_2)(Y_t)=(x_1,x_2)).
\end{aligned}
\right.
\end{equation}
In practice, the dynamics is discretized in time, and the biasing forces, 
$\Gamma^\alpha_t$, are approximated by empirical or time averages 
in each cells of a grid of the values of $(\xi_1,\xi_2)$. 
The bottom line of the method is to observe that, in the long-time 
limit, $(\Gamma^1_t,\Gamma^2_t)$ converges to $\nabla A$, where $A$ is the 
so-called free energy associated to $V$ and $(\xi_1,\xi_2)$ --- 
see~\cite{lelievre-rousset-stoltz-08}. At equilibrium, the potential eventually 
becomes $V-A\circ(\xi_1,\xi_2)$, which is such that the 
associated Boltzmann-Gibbs measure, proportional to $\exp(-\beta(V-A\circ
(\xi_1,\xi_2)))$, has uniform marginal laws along $
(\xi_1,\xi_2)$. Indeed, the free energy is defined, up to an additive constant, 
as:
\begin{equation}\label{eq:FE}
A(x_1,x_2) = - \beta^{-1} \ln \left( \int \exp(-\beta V(x_1,x_2,x_{3 \ldots n})) 
d x_{3 \ldots n} \right).
\end{equation}
Using the definition~\eqref{eq:FE} of the free energy $A$, it is easy to check 
that the equilibrium probability density for~\eqref{eq:ABF_standard}, namely
$$\psi_\infty(x_1,x_2,x_{3 \ldots n}) \varpropto \exp \Big( - \beta ( V
(x_1,x_2,x_
{3 \ldots n}) - A(x_1,x_2) ) \Big)$$ is such that the marginal law 
of $\psi_\infty$ along $(\xi_1,\xi_2)$ is uniform: $$\int_{\T^{n-2}} \psi_
\infty(x_1,x_2,x_{3 \ldots n}) d x_{3 \ldots n} = 1_{\T^2}.$$ Another related 
important property of the stochastic differential equation~\eqref
{eq:ABF_standard} 
is that the dynamics along $(\xi_1,\xi_2)$ has a simple diffusive 
behavior, since, by a straightforward It\^o calculus, for any test function $
\varphi: \T^2 \to \R$,
\begin{equation}\label{eq:diff_2d}
\partial_t \E(\varphi( (\xi_1,\xi_2) Y_t)) = \beta^{-1} \, \E( \Delta \varphi 
( (\xi_1,\xi_2) Y_t) )
\end{equation}
which is a weak form for the heat equation on the marginal law
along $(\xi_1,\xi_2)$ of the density of $Y_t$. Roughly speaking, the energy 
landscape has been flattened in the $(\xi_1,\xi_2)$-direction.

The aim of this work is to explore a generalization of the ABF 
method --- originally devised to compute a free-energy difference, which is a 
quantity of paramount importance in statistical 
mechanics~\cite{chipot-pohorille-07,lelievre-rousset-stoltz-10} --- 
focusing primarily on 
its adaptive importance sampling feature and with the 
objective of obtaining a diffusion along some chosen directions. 
More specifically, we have in mind the case of $m$ reaction 
coordinates with $m \ge 4$, for which the standard ABF approach cannot be used, 
because it would require that the biasing forces, 
namely $m$ functions of $m$ variables,
be approximated by a Monte Carlo procedure 
which is admittedly computationally prohibitive as $m$ 
increases. Moreover, the fact that the biased dynamics along the reaction 
coordinates is a simple diffusion in the whole torus $\T^m$ seems 
somewhat inappropriate, given that exploration of such a space 
may become extraordinarily long for large $m$.

The approach proposed herein consists in considering the 
dynamics:
\begin{equation} \label{eq:GABF_simp}
\left\lbrace
\begin{aligned}
&dX_t= -  \nabla \Big(  V - \sum_{\alpha=1}^2 A^\alpha_t \circ \xi_\alpha  \Big)
(X_t) \, dt + \sqrt{2 \beta^{-1}} dB_t,\\
&\text{for $\alpha=1,2$}, \  \frac{d A^\alpha_t}{d x_\alpha} (x_\alpha) = \E
(\partial_{x_\alpha} V(X_t) \, |\, \xi_\alpha(X_t)=x_\alpha).
\end{aligned}
\right.
\end{equation}
The interest of this dynamics is that only two one-dimensional functions have to 
be approximated. It is, therefore, expected that the Monte Carlo 
approximation of the biasing functions will be faster. One can 
check that this dynamics retains some essential 
features of the ABF dynamics~\eqref{eq:ABF_standard}, namely the fact that it 
leads to a simple diffusive behavior {\em in each direction} $
\xi_{\alpha}$ (see Section~\ref{sec:diff} below): for any test function $\varphi: 
\T \to \R$, and for $\alpha \in \{1,2\}$, $$\partial_t \E(\varphi( \xi_\alpha
(X_t))) = \beta^{-1} \, \E( \varphi'' (\xi_\alpha( X_t) ) ).$$ As a consequence, 
the marginal laws along $\xi_1$ and along $\xi_2$ of the 
equilibrium measure of the dynamics are uniform laws over $\T$. 
It ought to be noted, however, that~\eqref{eq:diff_2d} does not 
hold in general in such a situation, and that the marginal law of 
the equilibrium measure along $(\xi_1,\xi_2)$ is {\em not} in general 
a uniform law over $\T^2$. This is presented in detail in 
Section~\ref{sec:diff}.

A motivation for considering dynamics~\eqref{eq:GABF_simp} is that in the 
decoupled case, where
\begin{equation}\label{eq:decoupled_V}
\begin{aligned}
&V(x_1,x_2,x_{3\ldots n})= V(x_1) + V(x_2,x_{3\ldots n}) \text{ or }\\
&V(x_1,x_2,x_{3\ldots n})= V(x_2) + V(x_1,x_{3\ldots n}),
\end{aligned}
\end{equation}
then~\eqref{eq:GABF_simp} is equivalent to~\eqref{eq:ABF_standard}. The 
key idea here is that, if the two reaction coordinates, $\xi_1$ and $
\xi_2$ ,are ``not too strongly coupled'', then~\eqref{eq:GABF_simp} should be as 
effective as~\eqref{eq:ABF_standard}, at a 
far reduced cost.

In Section~\ref{sec:CV}, we propose a mathematical analysis of the long-time convergence of~\eqref{eq:GABF_simp}, which quantifies the underlying 
decoupling assumption. Section~\ref{sec:GABF} is devoted to a discussion of some 
generalization of the idea of the present work, in 
particular to the case of $m$ non-linear reaction coordinates with $m>2$. 
Finally, in Section~\ref{sec:num}, we report two numerical illustrations 
on non-trivial test cases, which illuminate the interest and the 
limitation of the approach.

\section{A convergence result}\label{sec:CV}

Let us introduce the Fokker-Planck equation associated to~\eqref{eq:GABF_simp}. 
Let us further refer $\psi(t,x)$ to as the density of the 
distribution of $X_t$. This function satisfies the partial differential 
equation:
\begin{equation} \label{eq:FP}
\left\lbrace
\begin{aligned}
&\partial_t \psi= \div \left( \nabla V \psi + \beta^{-1} \psi \right) - \partial_
{x_1} ( (A^1_t)' (x_1) \psi ) - \partial_{x_2} ( (A^2_t)' (x_2) \psi ),\\
&(A^1_t)' (x_1) =  \frac{\displaystyle \int \partial_{x_1} V (x) \psi(t,x) \, 
dx_2 \, dx_{3 \ldots n}} {\displaystyle\int  \psi(t,x) \, dx_2 \, dx_{3 \ldots 
n}},\\
&(A^2_t)' (x_2) =  \frac{\displaystyle\int \partial_{x_2} V (x) \psi(t,x) \, dx_1 
\, dx_{3 \ldots n}} {\displaystyle\int \psi(t,x) \, dx_1 \, dx_{3 \ldots n}}.
\end{aligned}
\right.
\end{equation}
where $(A^\alpha_t)'$ corresponds in what follows to the 
derivative of the one-dimensional function $x_\alpha \mapsto A^\alpha_t (x_
\alpha)$.

\subsection{Diffusive behavior}\label{sec:diff}

Let us introduce the marginal laws along $x_1$ and $x_2$ of $\psi$:
\begin{equation}\label{eq:marginals}
\psi^{x_1}(t,x_1)=\int  \psi(t,x) \, dx_2 \, dx_{3 \ldots n} \text{ and }
\psi^{x_2}(t,x_2)=\int  \psi(t,x) \, dx_1 \, dx_{3 \ldots n}.
\end{equation}
These marginal laws exhibit a simple diffusive behavior:
\begin{proposition}
The probability distribution functions $\psi^{x_1}$ and $\psi^{x_2}$ satisfy the 
heat equation: for $\alpha \in \{1,2\}$,
\begin{equation}\label{eq:diff}
\partial_t \psi^{x_\alpha} - \beta^{-1} \partial_{x_\alpha,x_\alpha} \psi^{x_
\alpha}= 0 \text{ on $\T$}.
\end{equation}
\end{proposition}
This property is easy to demonstrate by integrating the partial 
differential equation satisfied by $\psi$ in~\eqref{eq:FP} over the $x_i$, for $i 
\in \{1, \ldots n\} \setminus \{\alpha\}$.

As a simple consequence of~\eqref{eq:diff}, the 
marginal laws $\psi^{x_1}$ and $\psi^{x_2}$ converge to their equilibrium value 
$1_{\T}$ exponentially fast with rate $$r=4 \pi^2,$$ for example in the following 
relative entropy sense (see Definition~\ref{def:LSI} below):
\begin{equation}\label{eq:cv_marg}
\int \psi^{x_\alpha}(t,\cdot) \ln ( \psi^{x_\alpha}(t,\cdot) ) \le \int \psi^{x_
\alpha}(0,\cdot) \ln ( \psi^{x_\alpha}(0,\cdot) ) \exp(- 2 \beta^{-1} r t).
\end{equation}

\subsection{Stationary state}

If $A^1_t$ and $A^2_t$ reach a stationary state $A^1_\infty$ and $A^2_\infty$, it 
is standard that the stationary probability distribution function in~\eqref
{eq:FP} 
is: $$\psi_\infty(x) \varpropto \exp (-\beta ( V(x) -  A^1_\infty (x_1) - A^2_
\infty (x_2) ) ).$$ Thus, proving the existence of a stationary state is 
tantamount to proving the existence of a couple $(A^1_\infty,A^2_
\infty)$ solution to (note that the functions $A^\alpha_\infty$ are 
defined up to an additive function):
\begin{equation} \label{eq:FP_SS}
\left\lbrace
\begin{aligned}
(A^1_\infty)' (x_1) &=  \frac{\displaystyle \int \partial_{x_1} V (x) \exp (-
\beta 
( V(x) - A^2_\infty (x_2) ) )  \, dx_2 \, dx_{3 \ldots n}} {\displaystyle\int  
\exp (-\beta ( V(x)  - A^2_\infty (x_2) ) ) \, dx_2 \, dx_{3 \ldots n}},\\
(A^2_\infty)' (x_2) &=  \frac{\displaystyle\int \partial_{x_2} V (x) \exp (-\beta 
( V(x) -  A^1_\infty (x_1)  ) ) \, dx_1 \, dx_{3 \ldots n}} {\displaystyle\int 
\exp (-\beta ( V(x) -  A^1_\infty (x_1)  ) ) \, dx_1 \, dx_{3 \ldots n}}.
\end{aligned}
\right.
\end{equation}
Let us set $$\rho^\alpha (x_\alpha) = 
\exp(- \beta A^\alpha_\infty (x_\alpha) ).$$
Finding a solution to~\eqref{eq:FP_SS} is 
then equivalent to find a couple $(\rho^1,\rho^2)$ solution to 
(note that the functions $\rho^\alpha$ are defined up to a multiplicative 
constant):
\begin{equation} \label{eq:FP_SS_rho}
\left\lbrace
\begin{aligned}
\rho^1(x_1) &=  \int \frac{\displaystyle \exp (-\beta  V(x)) } {\displaystyle 
\rho^2(x_2)  } \, dx_2 \, dx_{3 \ldots n},\\
\rho^2(x_2) &=   \int \frac{\displaystyle \exp (-\beta  V(x)) } {\displaystyle 
\rho^1(x_1)   }\, dx_1 \, dx_{3 \ldots n}.
\end{aligned}
\right.
\end{equation}
\begin{proposition}
Let us assume that $V$ is a continuous function on $\T^n$. Then, there exists a 
solution to~\eqref{eq:FP_SS_rho}, and thus there exists a stationary state $
(\psi_
\infty,A^1_\infty,A^2_\infty)$ to~\eqref{eq:FP}.
\end{proposition}
\begin{proof}
Let us build a sequence of continuous functions $\rho^1_n: \T \to \R_+^*$ as ($
\rho^1_0=1$): $$\rho^1_{n+1}(x_1) = Z^1_{n+1}  \int \frac{\displaystyle \exp (-
\beta  A(x_1,x_2)) } {\displaystyle \int \frac{\displaystyle \exp (-\beta  A
(x_1,x_2)) } {\displaystyle \rho^1_n(x_1)   }\, dx_1  } \, dx_2 ,$$ where $Z^1_{n
+1}$ is chosen such that $\int 1 / \rho^1_{n+1}(x_1) \, dx_1=1$ and $A$ is the 
free energy~\eqref{eq:FE} introduced above. It is clear that if $(\rho^1_n)_{n 
\ge 0}$ converges in $L^\infty(\T)$ to $\rho^1_\infty$, then $\rho^1=\rho^1_
\infty$ and $\displaystyle \rho^2=\int \frac{\displaystyle \exp (-\beta  A
(x_1,x_2)) }  {\displaystyle \rho^1_\infty(x_1)   }\, dx_1$ is a solution to~
\eqref  {eq:FP_SS_rho}. We use the Arzel\`a--Ascoli theorem to show that 
$(\rho^1_n)_{n \ge 0}$ is a compact sequence in the space of 
real-valued
continuous functions over $\T$, endowed with the $L^\infty$-norm, which 
concludes the existence proof.

It first ought to be noted that since $V$ is continuous, there 
exists positive reals $a,b$ such that $0 < a \le \exp(-\beta A) \le b$ on $\T^2$. 
We, hence, have, for all $x_1 \in \T$,
\begin{equation*}
\int \frac{\displaystyle \exp (-\beta  A(x_1,x_2)) } {\displaystyle \int \frac
{\displaystyle \exp (-\beta  A(x_1,x_2)) } {\displaystyle \rho^1_n(x_1)   }\, 
dx_1  } \, dx_2
\ge  \int \frac{\displaystyle \exp (-\beta  A(x_1,x_2)) } {\displaystyle b \int  
\frac{\displaystyle 1 } {\displaystyle \rho^1_n(x_1)   }\, dx_1  } \, dx_2  \ge 
\frac{a}{b}.
\end{equation*}
Likewise,
\begin{equation*}
\int \frac{\displaystyle \exp (-\beta  A(x_1,x_2)) } {\displaystyle \int \frac
{\displaystyle \exp (-\beta  A(x_1,x_2)) } {\displaystyle \rho^1_n(x_1)   }\, 
dx_1  } \, dx_2 \le \frac{b}{a}.
\end{equation*}
>From this, one also obtains $  \frac{a}{b} \le Z^1_{n+1}  \le \frac{b}{a}$ which 
implies: $\forall n \ge 0$, $$\frac{a^2}{b^2} \le \rho^1_n \le \frac{b^2}{a^2} .
$$ The equicontinuity property remains to be checked to conclude the 
proof with the Arzel\`a--Ascoli theorem. This property is, however, 
straightforward to check by noting 
that $$|\rho^1_n(x_1) - \rho^1_n(x_1')| \le \frac{b^2}{a^2}  \int  |\exp (-\beta  
A(x_1,x_2)) - \exp (-\beta  A(x_1',x_2)) | \, dx_2. $$
\end{proof}

This altogether proves the existence of a stationary state. 
Its uniqueness is a consequence of the convergence result stated 
in the next section, which holds under an additional
weak-coupling assumption 
(see~\eqref{eq:lsi_large} below). 

A consequence of~\eqref{eq:FP_SS_rho}, which is also consistent 
with~\eqref{eq:diff}, is that the marginal probability density functions (see~\eqref{eq:marginals}) of $\psi_\infty$ along $x_1$ and $x_2$ are uniform:
$$\psi_\infty^{x_1}(x_1)=\int  \psi_\infty(x) \, dx_2 \, dx_{3 \ldots n}=1 \text
{ and } \psi_\infty^{x_2}(x_2)=\int  \psi_\infty(x) \, dx_1 \, dx_{3 \ldots n}
=1.$$

Last, it is worth noting that by and large $A^1_\infty$ and 
$A^2_\infty$ are not the free energies associated to $x_1$ and 
$x_2$ and defined as --- which should be compared with the 
definition of the bidimensional free energy~\eqref{eq:FE}:
\begin{equation}\label{eq:FE_1d}
\begin{aligned}
A^1(x_1) &= - \beta^{-1} \ln \left( \int \exp(-\beta V(x_1,x_2,x_{3 \ldots n})) 
dx_2 d x_{3 \ldots n} \right), \\
A^2(x_2) &= - \beta^{-1} \ln \left( \int \exp(-\beta V(x_1,x_2,x_{3 \ldots n})) 
dx_1 d x_{3 \ldots n} \right) .
\end{aligned}
\end{equation}
Actually, a special case for which $A^1_\infty=A^1$ and $A^2_
\infty=A^2$ (up to additive constants) is the decoupled case, namely if the 
two-dimensional free energy $A$ (defined by~\eqref{eq:FE}) writes as a sum of a 
function of $x_1$ and a function of $x_2$, which is equivalent, 
up to an additive constant, to:
\begin{equation}\label{eq:decoupled_A}
A(x_1,x_2)=A^1(x_1) + A^2(x_2).
\end{equation}
Under these premises, it is easy to check that $\rho^\alpha=\exp
(-\beta A^\alpha)$ is the unique solution to~\eqref{eq:FP_SS_rho},
up to a multiplicative constant. It ought to be 
noted that~\eqref{eq:decoupled_V} implies~\eqref{eq:decoupled_A}.

\subsection{Convergence}

Let us now consider a stationary state $(\psi_\infty,A^1_
\infty,A^2_\infty)$ to~\eqref{eq:FP}. The aim of this section is to prove the 
convergence of~\eqref{eq:FP} to this stationary state.

Let us first introduce the conditional probability density functions:
$$\psi_{\infty |x_1} (x_2, x_{3 \ldots n} ) = \frac{\psi_\infty (x_1,x_2,x_{3 
\ldots n} )}{\psi_\infty^{x_1}(x_1)}=\psi_\infty (x_1,x_2,x_{3 \ldots n} )$$
and $$\psi_{\infty |x_2} (x_1, x_{3 \ldots n} ) = \frac{\psi_\infty (x_1,x_2,x_{3 
\ldots n} )}{\psi_\infty^{x_2}(x_2)}=\psi_\infty (x_1,x_2,x_{3 \ldots n} ).$$
In our particular case, this consists only in freezing one variable
of $\psi_\infty$.

To achieve this objective, we will need tools related to 
logarithmic Sobolev inequalities. We recall that (see~\cite{ABC-00,arnold-markowich-toscani-unterreiter-01,villani-03}):

\begin{definition}\label{def:LSI}
A probability measure $\nu$ is said to satisfy a logarithmic Sobolev inequality 
with constant $\rho>0$ (in short: LSI($\rho$)) if for all probability measures 
$\mu$ such that $\mu$ is absolutely continuous with respect to $\nu$ (denoted $
\mu \ll \nu$ in the following), $$H(\mu|\nu) \leq \frac{1}{2 \rho} I (\mu | \nu),
$$ where $$H(\mu | \nu)=\int \ln \left(\frac{d\mu}{d\nu} \right) d \mu$$ is the 
relative entropy of $\mu$ with respect to $\nu$ and $$I(\mu | \nu)=\int
\left|\nabla \ln \left(\frac{d\mu}{d\nu} \right)\right|^2 d \mu$$ is the 
so-called Fisher information of $\mu$ with respect to $\nu$.
\end{definition}

Since the measures $\psi_{\infty | x_1}(x_2, x_{3 \ldots n}) \, dx_2 d x_{3 
\ldots n}$ and $\psi_{\infty | x_2}(x_1, x_{3 \ldots n}) \, dx_1 d x_{3 \ldots n}
$ are defined on a compact space ($\T^{n-1}$) and since $\psi_\infty$ is smooth, 
it follows from standard arguments --- e.g. Holley-Stroock 
criterion, that they satisfy logarithmic Sobolev inequalities. 
Let us introduce the associated positive constants $\rho_1$ and $\rho_2$: 
\begin{equation}\label{eq:lsi}
\begin{aligned}
&\text{ $\psi_{\infty | x_1}(x_2, x_{3 \ldots n}) \, dx_2 d x_{3 \ldots n}$ 
(resp. $\psi_{\infty |     x_2}(x_1, x_{3 \ldots n}) \, dx_1 d x_{3 \ldots n}$) 
satisfies}\\
&\text{a LSI with constant $\rho_1$ (resp. $\rho_2$) for all $x_1 \in \T$ (resp. 
$x_2 \in \T$).}
\end{aligned}
\end{equation}
We also need to introduce the coupling constants $$\kappa_1=\| 
\nabla_{x_2,x_{3 \ldots n}} (\partial_{x_1} V ) \|_{L^\infty(\T^n)} \text{ and } 
\kappa_2=\| \nabla_{x_1,x_{3 \ldots n}} (\partial_{x_2} V ) \|_{L^\infty(\T^n)}$$
which are well defined since $V$ is assumed to be smooth over $\T^n$. It is worth 
noting that $\kappa_1=0$ or $\kappa_2=0$ is equivalent 
to~\eqref{eq:decoupled_V}, which implies the relation~\eqref{eq:decoupled_A}. 
This motivates the terminology of coupling constants.

For the convergence result to hold, we need the coupling constants to be 
sufficiently small compared to the logarithmic Sobolev constants $\rho_1$ and $
\rho_2$:
\begin{equation}\label{eq:lsi_large}
\rho_1 \rho_2 \ge \beta^2 \kappa_1 \kappa_2.
\end{equation}
We are now in a position where we can state the main mathematical result 
of this contribution.
\begin{theorem}\label{th:CV}
Let us assume~\eqref{eq:lsi_large}. The probability density
function $\psi(t,\cdot)$ then
converges to $\psi_\infty$ exponentially fast: for any $\varepsilon \in 
(0,\lambda)$, $\exists C >0$, $\forall t \ge 0$
\begin{equation}\label{eq:CV_L1}
\int_{\T^n} |\psi(t,x) -\psi_\infty(x)| \, dx \le C \exp \bigg(- \beta^{-1} \min
\Big( (\lambda - \varepsilon) , r \Big) t \bigg)
\end{equation}
where $r=4 \pi^2$ and
\begin{equation}\label{eq:lambda}
\lambda=\frac{\rho_1 + \rho_2 - \sqrt{(\rho_1- \rho_2)^2 + \frac{ 4 \kappa_1 
\kappa_2 }{\rho_1 \rho_2} }}{4}
\end{equation}
is a positive constant. Furthermore, for any positive time $t_0$ 
and  $\varepsilon \in (0,\lambda)$, $\exists \bar{C} >0$, $\forall t \ge t_0$,
\begin{equation}\label{eq:CV_A}
\int_{\T} |(A^\alpha_t)' - (A^\alpha_\infty)'|^2  \le \bar{C} \exp\bigg(-2 \beta^
{-1} \min\Big( (\lambda - \varepsilon) , r \Big) t \bigg),
\end{equation}
where $\alpha \in \{1,2\}$.
\end{theorem}
The interpretation of this theorem is that, if the coupling constants $\kappa_1$ 
and $\kappa_2$ are sufficiently small, the dynamics converges exponentially fast 
with a rate essentially limited by $\displaystyle\frac{\min(\rho_1,\rho_2)}{2}$ 
(namely $\lambda$ when $\kappa_1=0$ or $\kappa_2=0$). This constant is expected 
to be larger than the logarithmic Sobolev constant of the original measure $\mu$ 
(which gives the rate of convergence of the original dynamics~\eqref{eq:Q}) if $
\xi_1$ and $\xi_2$ are well chosen {---} see related discussions 
in the work~\cite{lelievre-rousset-stoltz-08,lelievre-09}.

\begin{proof}
The proof is an adaptation of the proof for the long-time 
convergence of the ABF process, which can be found in~\cite
{lelievre-rousset-stoltz-08}. It can be assumed without loss of 
generality that $\beta=1$ up to the following change of variable: $\tilde{t}=
\beta^{-1} t$, $\tilde{\psi}(\tilde{t},x)=\psi(t,x)$ and $\tilde{V}(x)=\beta
V(x)$.

Let us first rewrite the partial differential equation satisfied by $\psi$ as:
$$\partial_t \psi = \div \left( \psi \, \nabla \left( \frac{\psi}{\psi_\infty} 
\right) \right) + \partial_{x_1} \left( ( (A^1_\infty)' - (A^1_t)' ) \psi \right) 
+ \partial_{x_2} \left( ( (A^2_\infty)' - (A^2_t)' ) \psi \right).$$ Let us 
consider {then} the relative entropies $$E(t)= H(\psi | \psi_
\infty) = \int_{\T^n} \psi(t,\cdot) \ln (\psi(t,\cdot) / \psi_\infty)$$
and, for $\alpha \in \{1,2\}$, $$E^\alpha_M(t)= H(\psi^{x_\alpha} | \psi^{x_
\alpha}_\infty) = \int_{\T} \psi^{x_\alpha}(t,\cdot) \ln (\psi^{x_\alpha}(t,
\cdot)).$$ The aim of the present proof is to show that $E(t)$ 
converges to $0$ exponentially fast. It is already clear from~\eqref{eq:diff} 
that $E^\alpha_M(t)$ converges to zero exponentially fast (see~\eqref
{eq:cv_marg}), so that it is enough to consider $$E_m^\alpha(t)=E(t) - E^\alpha_M
(t)= \int_{\T} H (\psi_{|x_\alpha}(t, \cdot) | \psi_{\infty|x_\alpha} ) \, \psi^
{x_\alpha}(t,x_\alpha) d x_\alpha,$$ where $$H (\psi_{|x_\alpha}(t, \cdot) | 
\psi_{\infty|x_\alpha} ) = \int_{\T^{n-1}} \psi_{|x_\alpha}(t, \cdot) \ln ( \psi_
{|x_\alpha}(t, \cdot) / \psi_{\infty |x_\alpha} ) $$ is the relative entropy, 
with respect to $\psi_{\infty|x_\alpha}$, of the conditional probability density 
functions $$\psi_{|x_1} (t,x_2, x_{3 \ldots n} ) = \frac{\psi (t, x_1,x_2,x_{3 
\ldots n} )}{\psi^{x_1}(t,x_1)} \text{ and } \psi_{|x_2} (t,x_1, x_{3 \ldots n} ) 
= \frac{\psi (t,x_1,x_2,x_{3 \ldots n} )}{\psi^{x_2}(t,x_2)}.$$ 
Let us focus on the case $\alpha=1$, albeit
similar computations hold for $\alpha=2$. Let 
us compute
\begin{align}
\frac{d E^1_m}{dt} &= \frac{d E}{dt} - \frac{d E^1_M}{dt},\nonumber\\
& = - \int_{\T^n} \left| \nabla \ln \left( \frac{\psi}{\psi_\infty} \right)
\right|^2 \psi + \sum_{\gamma=1}^2 \int_{\T^n} \left((A^\gamma_t)' -( A^\gamma_
\infty)' \right) \partial_{x_\gamma} \ln \left(
\frac{\psi}{\psi_\infty} \right) \psi \label{eq:E1} \\
& \quad + \int_{\T} \left| \partial_{x_1} \ln
\left(\psi^{x_1} \right)
\right|^2 \psi^{x_1}.\nonumber
\end{align}
It is rather easy to check the following identity:
\begin{equation}\label{eq:A_A}
(A^\alpha_t)'-(A^\alpha_\infty)' = \int_{\T^{n-1}} \partial_{x_\alpha} \ln \left
(\frac{\psi}{\psi_\infty} \right)
\frac{\psi}{\psi^{x_\alpha}} \, dx_{\overline{\alpha}} \, dx_{3 \ldots n} - 
\partial_{x_\alpha} \ln \left(\psi^{x_\alpha}\right),
\end{equation}
where $\overline{\alpha}=1$ (resp. $\overline{\alpha}=2$) when $\alpha=2$ (resp. 
$\alpha=1$). Using~\eqref{eq:A_A} in~\eqref{eq:E1}, we obtain
\begin{align}
\frac{d E^1_m}{dt}&= - \int_{\T^n} \left| \partial_{x_2,x_{3 \ldots n}} \ln \left
( \frac{\psi}{\psi_\infty} \right)
\right|^2 \psi \nonumber \\
&\quad - \int_{\T^n} \left| \partial_{x_1} \ln \left( \frac{\psi}{\psi_\infty} 
\right)
\right|^2 \psi + \int_{\T} \left( \int_{\T^{n-1}} \partial_{x_1} \ln \left(
\frac{\psi}{\psi_\infty} \right) \psi \, dx_2 \, dx_{3 \ldots n} \right)^2
\frac{1}{\psi^{x_1}} \, d{x_1} \nonumber \\
& \quad - \int_{\T^n}  \partial_{x_1} \ln \left(\psi^{x_1} \right)\partial_{x_1} 
\ln \left(
\frac{\psi}{\psi_\infty} \right) \psi + \int_{\T} \left| \partial_{x_1} \ln
\left( \psi^{x_1} \right)
\right|^2 \psi^{x_1}\nonumber \\
& \quad + \int_{\T^n} \left((A^2_t)' -( A^2_\infty)' \right) \partial_{x_2} \ln 
\left(
\frac{\psi}{\psi_\infty} \right) \psi. \nonumber
\end{align}
By virtue of
the Cauchy-Schwarz inequality, the term on the second line is non-positive. 
Using again~\eqref{eq:A_A}, we, hence, have
\begin{align}
\frac{d E^1_m}{dt}& \leq - \int_{\T^n} \left| \partial_{x_2,x_{3 \ldots n}} \ln 
\left( \frac{\psi}{\psi_\infty} \right)
\right|^2 \psi - \int_{\T}  \partial_{x_1} \ln
\left( \psi^{x_1} \right)   \psi^{x_1}  \left( (A^1_t)'-(A^1_\infty)' \right) 
\nonumber \\
& \quad + \int_{\T^n} \left((A^2_t)' -( A^2_\infty)' \right) \partial_{x_2} \ln 
\left(
\frac{\psi}{\psi_\infty} \right) \psi. \label{eq:E3}
\end{align}
We now need an estimate for $|(A^1_t)'-(A^1_\infty)'|$. For any coupling measure 
$\pi \in \Pi(\psi_{|x_1}(t,\cdot) , \psi_{\infty|x_1})$, it holds:
\begin{align*}
&|(A^1_t)'(x_1)-(A^1_\infty)'(x_1)|\\
&= \left|\int_{\T^{n-1} \times \T^{n-1}} \left( \partial_{x_1} V (x_1,x_2,x_{3 
\ldots n}) - \partial_{x_1} V (x_1,x_2',x_{3 \ldots n}') \right) \,\pi( dx_2 \, 
dx_{3 \ldots n} , \, dx_2' \, dx_{3 \ldots n}')
\right|\\
& \leq \|\nabla_{x_2,x_{3 \ldots n}} (\partial_{x_1} V )   \|_{L^\infty} \int_{\R 
\times \R} |(x_2,x_{3 \ldots n})-(x_2',x_{3 \ldots n}')| \,\pi( dx_2 \, dx_{3 
\ldots n} , \, dx_2' \, dx_{3 \ldots n}')\\
& \leq \kappa_1 \int_{\R \times \R} |(x_2,x_{3 \ldots n})-(x_2',x_{3 \ldots n}')| 
\,\pi( dx_2 \, dx_{3 \ldots n} , \, dx_2' \, dx_{3 \ldots n}').
\end{align*}
Taking the infimum over all $\pi \in \Pi(\mu_{t,x},\mu_{\infty,x})$, we 
obtain
\begin{equation*}
|(A^1_t)'(x_1)-(A^1_\infty)'(x_1)| \leq \kappa_1
W(\psi_{|x_1}(t,\cdot) , \psi_{\infty|x_1})
\end{equation*}
where $W$ stands for the ($L^1$) Wasserstein distance. We 
will now resort to the fact that if $
\nu$ is a probability measure satisfying a logarithmic Sobolev inequality with 
constant $\rho$, then we have the Talagrand inequality (see~\cite{bobkov-gotze-99,otto-villani-00}): 
For all probability measures $\mu$ such that $\mu 
\ll \nu$, $$W(\mu,\nu) \leq \sqrt{\frac{2}{\rho} H (\mu | \nu)}.$$
Using~\eqref{eq:lsi} together with the Talagrand inequality, 
we, thus, obtain (the 
proof being evidently 
similar for $\alpha=2$): for $\alpha \in \{1,2\}$,
\begin{equation}\label{eq:estim_A-A}
|(A^\alpha_t)'(x_\alpha)-(A^\alpha_\infty)'(x_\alpha)| \leq  \kappa_\alpha \sqrt
{ \frac{2}{\rho_\alpha} H(\psi_{|x_\alpha}(t,\cdot) | \psi_{\infty|x_\alpha} ) }.
\end{equation}
Using this estimate in~\eqref{eq:E3}, the constants $\rho_1$ and $\rho_2$ 
introduced in~\eqref{eq:lsi}, and Cauchy-Schwarz and Young inequalities, we get:
\begin{align*}
\frac{d E^1_m}{dt}
& \leq - \frac{1}{2} \int_{\T^n} \left| \partial_{x_2,x_{3 \ldots n}} \ln \left
( \frac{\psi}{\psi_\infty} \right)
\right|^2 \psi  +  \sqrt{\int_{\T} \left| (A^1_t)'-(A^1_\infty)' \right|^2
\psi^{x_1}} \sqrt{\int_{\T} \left| \partial_{x_1} \ln
  \left(\psi^{x_1} \right)
\right|^2 \psi^{x_1}}\\
& \quad + \frac{1}{2} \int_{\T} \left| (A^2_t)'-(A^2_\infty)' \right|^2
\psi^{x_2}\\
 & \leq -  \rho_1  E_m^1 + \kappa_1  \sqrt{
  \frac{2}{\rho_1} E_m^1 } \sqrt{I(\psi^{x^1} | \psi^{x_1}_\infty)} + \frac{1}{2} 
\kappa_2^2 \frac{2}{\rho_2} E_m^2.
\end{align*}
Employing~\eqref{eq:diff}, it is standard to show that (see \cite[Lemma 12]{lelievre-rousset-stoltz-08} for example):
$$I(\psi^{x^1}(t,\cdot) | \psi^{x_1}_\infty)  \le I_0 \exp ( - 2 rt )$$
where $I_0 = I(\psi^{x^1}(0,\cdot) | \psi^{x_1}_\infty)$ and,
we recall, $r=4 \pi^2$. 
We finally find for any positive~$\eta_1$:
\begin{align*}
\frac{d E^1_m}{dt}  & \leq -  \rho_1 (1 - \eta_1) E_m^1 + \frac{\kappa_2^2}
{\rho_2} E_m^2 + \frac{\kappa_1^2}{2 \rho_1^2 \eta_1} I_0 \exp(-2rt).
\end{align*}
Utilizing a similar reasoning with $\alpha=2$, 
the following system of 
inequalities is obtained, wherein
$\eta_1,\eta_2$ are positive real numbers to be fixed:
\begin{equation*}
\left\{
\begin{aligned}
\frac{d E^1_m}{dt}  & \leq -  \rho_1 (1 - \eta_1) E_m^1 + \frac{\kappa_2^2}
{\rho_2} E_m^2 + \frac{\kappa_1^2}{2 \rho_1^2 \eta_1} I_0 \exp(-2rt),\\
\frac{d E^2_m}{dt}  & \leq -  \rho_2 (1 - \eta_2) E_m^2 + \frac{\kappa_1^2}
{\rho_1} E_m^1 + \frac{\kappa_2^2}{2 \rho_2^2 \eta_2} I_0 \exp(-2rt).
\end{aligned}
\right.
\end{equation*}
In the limit $\eta_1=\eta_2=0$, we get the linear system:
\begin{equation*}
\left\{
\begin{aligned}
\frac{d u^1}{dt}  & \leq -  \rho_1 u^1 + \frac{\kappa_2^2}{\rho_2} u^2,\\
\frac{d u^2}{dt}  & \leq -  \rho_2 u^2 + \frac{\kappa_1^2}{\rho_1} u^1,
\end{aligned}
\right.
\end{equation*}
for which it can be be shown quite simply
that, under the assumption~\eqref{eq:lsi_large},
$$\forall t \ge 0, \, 
\|(u^1,u^2)(t) \| \le \|(u^1,u^2)(0) \| \exp(-2 \lambda t)$$
where $\lambda$ is defined by~\eqref{eq:lambda}. It is then easy to 
reach the result~\eqref{eq:CV_L1}, 
employing the Csisz\'ar-Kullback inequality:
\begin{equation}\label{eq:CK}
\int | \psi - \psi_\infty  | \leq \sqrt{2 H(\psi | \psi_\infty)}
\end{equation}
and the fact that $H(\psi | \psi_\infty)=E=E_M^\alpha + E_m^\alpha$.
We refer the reader,
for example,
to the end of the proof of Theorem~1 in~\cite{lelievre-minoukadeh-10} for a similar reasoning. 

Finally, the convergence results~\eqref{eq:CV_A} on $(A_t^1)'$ and $(A_t^2)'$ are 
easily obtained from~\eqref{eq:estim_A-A} and the fact that (using~\eqref
{eq:diff}) $\psi^{x_1}$ and $\psi^{x_2}$ are bounded from below by a positive 
constant for times larger than any arbitrary small positive time, see the 
beginning of Section~3.3.2 in~\cite{lelievre-rousset-stoltz-08} for more details.
\end{proof}

\section{Discussion and generalizations}\label{sec:GABF}

\subsection{The case of nonlinear reaction coordinates}

We now would like to discuss generalizations of the approach 
introduced above, in the 
case where $(\xi_1,\xi_2)$ are 
not simply $(x_1,x_2)$, which constitutes the vast majority of
practical situations.

Let us denote $\xi_1:{\mathcal D}  \to \T$,
and $\xi_2:{\mathcal D}  \to \T$, the chosen variables
forming the multidimensional
reaction coordinate. A natural 
generalization of~\eqref{eq:GABF_simp} is the following:
\begin{equation} \label{eq:GABF}
\left\lbrace
\begin{aligned}
&dX_t=- \nabla \left(V - \sum_{\alpha=1}^2 A^\alpha_t \circ \xi_\alpha \right) 
(X_t) \, dt + \sqrt{2 \beta^{-1}} dB_t,\\
&\text{for $\alpha=1,2$}, \,\frac{d A^\alpha_t}{d z_\alpha} (z_\alpha) = \E(f^
\alpha_t(X_t) | \xi_\alpha(X_t)=z_\alpha),
\end{aligned}
\right.
\end{equation}
where $f^\alpha_t$ stand for the
so-called {\em local mean forces} associated to $\xi_\alpha$ and defined by:
\begin{equation}
\label{eq:f1}
f^1_t=\left( \frac{\nabla (V - A^2_t \circ \xi_2) \cdot \nabla \xi_1}{|\nabla 
\xi_1|^{2}}
  - \beta^{-1} \div\left(\frac{\nabla \xi_1}{|\nabla \xi_1|^{2}}\right) \right)
\end{equation}
and
\begin{equation}
\label{eq:f2}
f^2_t=\left( \frac{\nabla (V - A^1_t \circ \xi_1) \cdot \nabla \xi_2}{|\nabla 
\xi_2|^{2}}
  - \beta^{-1} \div\left(\frac{\nabla \xi_2}{|\nabla \xi_2|^{2}}\right) \right).
\end{equation}
It ought to be noted 
that~\eqref{eq:GABF} reduces to~\eqref{eq:GABF_simp} in the specific case 
of $\xi_\alpha(x)=x_\alpha$.

If $\psi$ denotes 
the probability density function of $X_t$, then the marginal 
probability density functions are:
\begin{equation*}
\psi^{\xi_\alpha}(t,z_\alpha)=\int_{\Sigma^\alpha(z_\alpha)} \psi(t,x) \, \delta_
{\xi_\alpha(x)-z_\alpha} (dx) \end{equation*}
which boils down to~\eqref{eq:marginals} in the case of $\xi_\alpha(x) =
x_\alpha$.

The conditional measure $\delta_{\xi^\alpha(x)-z_\alpha} (dx)$ 
obeys the following definition: 
For any test function $\varphi: {\mathcal D} \to \R$,
$$\int_{\mathcal D} \varphi(x) \, dx = \int_{\T} \int_{\Sigma^\alpha(z_\alpha)} 
\varphi(x) \, \delta_{\xi^\alpha(x)-z_\alpha}(dx)\, dz_\alpha,$$
where $\Sigma^\alpha(z_\alpha)=\{ x \in {\mathcal D}, \xi^\alpha(x)=z_\alpha \}$.
A corollary of the coarea formula is that:
\begin{equation}\label{eq:der_z}
\partial_{z_\alpha} \psi^{\xi_\alpha}= \int_{\Sigma^\alpha(z_\alpha)} \left
( \frac{\nabla \psi \cdot \nabla \xi_\alpha}{|\nabla \xi_\alpha|^2} +  \div \left
(\frac{\nabla \xi_\alpha}{|\nabla \xi_\alpha|^2} \right) \psi \right)\, \delta_
{\xi_\alpha(x) - z} (dx).
\end{equation}
For detailed proofs and references, the reader 
is referred to~\cite[Lemma 3.10]{lelievre-rousset-stoltz-10}.

One can check the following property, reminiscent of~\eqref{eq:diff}:
\begin{proposition}
Let us assume that
\begin{equation}\label{eq:cst_grad_xi}
|\nabla \xi_1| = |\nabla \xi_2| = 1.
\end{equation}
The probability distribution functions $\psi^{\xi_1}$ and $\psi^{\xi_2}$ satisfy 
the heat equation: for $\alpha \in \{1,2\}$,
\begin{equation}\label{eq:diff_gen}
\partial_t \psi^{\xi_\alpha} - \beta^{-1} \partial_{z_\alpha,z_\alpha} \psi^{\xi_
\alpha}= 0 \text{ on $\T$}.
\end{equation}
\end{proposition}
\begin{proof}
Let us prove~\eqref{eq:diff_gen} for $\alpha=1$, the proof being 
evidently
similar for $\alpha=2$. For any test function $\varphi: \T \to \R$:
\begin{align*}
\int_{\T} \partial_t \psi^{\xi_1} \varphi 
&= \int_{\mathcal D} \partial_t \psi  \, \varphi\circ \xi_1 \\
&= \int_{\mathcal D} \div \Big(( \nabla (V - A^1_t \circ \xi_1 - A^2_t \circ 
\xi_2) \psi + \beta^{-1} \nabla \psi) \Big) \varphi\circ \xi_1 \\
&= - \int_{\mathcal D} \Big(  ( \nabla (V  - A^1_t \circ \xi_1- A^2_t \circ 
\xi_2) \psi + \beta^{-1} \nabla \psi) \Big) \cdot \nabla \xi_1 \, \varphi' \circ 
\xi_1 \\
&= - \int_{\mathcal D}   \nabla (V - A^2_t \circ \xi^2) \cdot \nabla \xi_1 \, 
\psi\, \varphi' \circ \xi_1
+ \int_{\mathcal D} (A^1_t)' \circ \xi_1 \psi \, \varphi' \circ \xi_1 \\
& \quad - \beta^{-1}  \int_{\mathcal D}  \nabla \psi \cdot \nabla \xi_1 \,  
\varphi' \circ \xi_1 \\
&= - \int_{\mathcal D}  \nabla  (V - A^2_t \circ \xi^2) \cdot \nabla \xi_1 \, 
\psi\, \varphi' \circ \xi_1
+ \int_{\mathcal D} f^1_t \psi\, \varphi' \circ \xi_1 \\
& \quad - \beta^{-1}  \int_{\mathcal D} \nabla \psi \cdot \nabla \xi_1 \,  
\varphi' \circ \xi_1 \\
&= - \beta^{-1} \int_{\mathcal D} \Delta \xi_1 \, \psi\, \varphi' \circ \xi_1- 
\beta^{-1} \int_{\mathcal D} \nabla \psi \cdot \nabla \xi_1 \,  \varphi' \circ 
\xi_1 \\
&= - \beta^{-1}  \int_{\T} \partial_{z_1} \psi^{\xi_1} \varphi',
\end{align*}
where we used~\eqref{eq:der_z} for the last equality. This is, 
indeed, a weak formulation of~\eqref{eq:diff_gen}.
\end{proof}

Assumption~\eqref{eq:cst_grad_xi} essentially amounts to assuming that $\xi_
\alpha$ is the signed distance to $\Sigma^\alpha(0)$. It is straightforward to 
generalize the above result by changing the assumption~\eqref{eq:cst_grad_xi} to: 
$|\nabla \xi_1|$ and $|\nabla \xi_2|$ are constant functions. It is also possible 
to generalize it to the case where
$|\nabla \xi_1|$ (resp. $|\nabla \xi_2|$) depends on 
$x$ only through $\xi_1(x)$ (resp. $\xi_2(x)$) with slight modifications of the 
definition of the functions $f^\alpha_t$.
Generalization of the convergence results of Theorem~\ref{th:CV} to this setting 
is also possible, yet we will not pursue in this direction here.

\subsection{The multi-dimensional setting}

When more than two variables describing the multidimensional reaction coordinate
are needed, the biasing procedure introduced above can be
generalized as follows:
\begin{itemize}
\item[(i)] In the case of $m$ collective variables $\xi_1, \ldots, \xi_m$, a 
natural generalization would consist in biasing the dynamics
by a potential $A^1_t\circ \xi_1 + \ldots + A^m_t \circ \xi_m$,
using a straightforward extension of the dynamics~\eqref{eq:GABF}.
It is worth noting
that the complexity is typically linear in~$m$, whereas it is exponential 
in $m$ for a standard ABF approach. For another biasing approach,
in the context 
of multiple reaction coordinates, the
reader is referred to~\cite{piana-laio-07}.
\item[(ii)] In the case of $m$ collective variables, one 
might want to keep certain collective
variables coupled, e.g. $\xi_1$ and $\xi_2$.
A natural way to build a biased dynamics in 
such a case consists in considering
a biasing potential $A^{1,2}_t\circ(\xi_1,\xi_2) + 
A^3_t\circ \xi_3 + \ldots + A^m_t \circ \xi_m$, and using
an adequate formula for updating $A^{1,2}_t$ based on the standard ABF 
approach~\eqref{eq:ABF_standard} (see~\cite[Section 5.1.1]{lelievre-rousset-stoltz-10} for example for adequate formulae 
for general~$(\xi_1,\xi_2)$).
\end{itemize}

As mentioned above and roughly speaking, biasing the potential by the sum $A^1_t 
\circ \xi_1 +
\ldots + A^m_t \circ \xi_m$ rather than by a $m$-dimensional adaptive potential
$A_t \circ (\xi_1, \ldots, \xi_m)$ --- which would yield a 
perfect diffusion along the $m$-dimensional vector $(\xi^1, \ldots, \xi^m)(X_t)$
--- is tantamount to supposing some sort
of decoupling on the collective variables $\xi^1, \ldots, \xi^m$. More
precisely, if the free energy associated to $(\xi_1, \ldots, \xi_m)$ writes as
a sum of functions $A^1 \circ \xi_1 + \ldots + A^m \circ \xi_m$, then the two
algorithms (generalized-ABF and ABF) are equivalent. 
We, therefore, expect the method to be
efficient if the collective variables are loosely coupled,
as would be the case for two dihedral
angles distant from each other in a molecule. 
Conversely, should the two collective variables
be more strongly coupled, it might
be interesting to resolve this
coupling in the adaptive potential, as discussed in
item (ii) above.

\subsection{Further generalizations}

In practice, it may be difficult to implement the dynamics~\eqref{eq:GABF}, in 
particular on account
of the computations of the analytical expressions for $f^
\alpha_t$ which may be cumbersome. A natural idea, which is, however,
not supported 
by any mathematical reasoning, 
consists in simplifying the expressions for $f^\alpha_t$, 
considering, for example:
\begin{equation}
\label{eq:f1_simp_1}
f^1=\left( \frac{\nabla V \cdot \nabla \xi_1}{|\nabla \xi_1|^{2}}
  - \beta^{-1} \div\left(\frac{\nabla \xi_1}{|\nabla \xi_1|^{2}}\right) \right),
\end{equation}
or even
\begin{equation}
\label{eq:f1_simp_2}
f^1= \frac{\nabla V  \cdot \nabla \xi_1}{|\nabla \xi_1|^{2}},
\end{equation}
and similar formulae for $f^2$. 
A very practical approach {can be stated as follows: {\em If}} 
$A^1_t$ and $A^2_t$ happen
to converge to some functions $A^1_\infty$ and $A^2_
\infty$, then it is always possible to fix this bias, and then 
to use unbiasing procedures {akin to} 
those described in Section~\ref{sec:FE} below --- see~\eqref
{eq:unbias}, to obtain canonical averages.

\section{Illustrations for {free-energy calculations}}\label
{sec:num}

In this section, we illustrate the interest and the limitation of the approach 
in two test cases, using the {\sc Namd} simulation package~\cite
{Phillips2005}. Propagation of motion is performed employing 
Langevin dynamics, in lieu of overdamped Langevin. 
Estimates of the biasing force rely upon trajectory, time 
averages rather than empirical averages over many replicas. Use 
is made of expression~\eqref{eq:f1_simp_1} to determine the local mean force. 
Overall, the method utilized here is similar to 
the biasing techniques~\eqref{eq:GABF_simp}--\eqref{eq:GABF} 
described above.

\subsection{Recovering the free energy}\label{sec:FE}

To recover the free energy associated to some of the reaction coordinates chosen 
to bias the dynamics, one can simply use a grid of the values of the reaction 
coordinates, and classical formulas for the free energy or
its derivative. This relies on the fact that from~\eqref{eq:GABF_simp}--\eqref
{eq:GABF}, if
$((A^1_t)',(A^2_t)')$ reaches an equilibrium $((A^1_\infty)',(A^2_\infty)')$, 
then the law of~$X_t$ at equilibrium is
proportional to $\exp(-\beta (V - A^1_\infty \circ \xi_1 - A^2_\infty \circ
\xi_2))$.

Let us be more precise, considering~\eqref{eq:GABF}. To get the free energy
$A(z_1,z_2)$ {associated} to the {bidimensional} 
reaction coordinate $(\xi_1,\xi_2)$, one can use the formula (compare with~\eqref
{eq:FE} which is~\eqref{eq:A} in the simple case $(\xi_1,\xi_2)=(x_1,x_2)$):
\begin{equation}\label{eq:A}
A (z_1,z_2) = - \beta^{-1} \ln \int \exp(-\beta V(x)) \delta_{(\xi_1,\xi_2)(x)-
(z_1,z_2)}(dx).
\end{equation}
It is {indeed} easy to {verify} that, at 
equilibrium, 
\begin{equation}\label{eq:recover}
\lim_{\varepsilon \to 0} - \beta^{-1} \ln \E ( \delta^\varepsilon
\left( (\xi_1,\xi_2)(X_t) - (z_1,z_2) \right)) +A^1_\infty (z_1) + A^2_\infty
(z_2) = A (z_1,z_2)
\end{equation}
up to an additive constant. Here, $\delta^\varepsilon$ 
denotes an {approximation of identity} 
converging to a Dirac mass at $0$ when $\varepsilon$ goes to $0$. In practice, 
{piecewise} constant approximation {is} obtained 
over a grid of values {accessible to} the reaction coordinates.

Another interesting formula to compute $A(z_1,z_2)$ {derives} from 
the formula:
\begin{equation}\label{eq:grad_A}
\nabla A (z_1,z_2) = \E_\mu \Big( F (X) | (\xi_1,\xi_2)(X) = (z_1,z_2) \Big)
\end{equation}
where the notation $\E_\mu$ means that $X$ is distributed according to $\mu$
and $F$ is a two-dimensional vector defined by: $\forall \alpha \in \{1,2\}$,
\begin{equation}
\label{eq:F}
F_\alpha = \sum_{\gamma=1}^2 G^{-1}_{\alpha,\gamma} \nabla \xi_\gamma
\cdot \nabla V   - \beta^{-1} \div\left(\sum_{\gamma=1}^2 G^{-1}_{\alpha,\gamma} 
\nabla \xi_\gamma\right),
\end{equation}
where $G^{-1}_{\alpha,\gamma}$ denotes the $(\alpha,\gamma)$-component of
the inverse of the matrix with components: $\forall \alpha,\gamma \in \{1, 2\}$,
\begin{equation}\label{eq:G}
G_{\alpha,\gamma}=\nabla \xi_{\alpha} \cdot \nabla \xi_{\gamma}.
\end{equation}
{It} is easy to check that if the time marginal 
{law} of $X_t$ solution to~\eqref{eq:GABF} reaches an equilibrium, then 
\begin{equation}\label{eq:cond_avg}
\E \Big( F (X_t) | (\xi_1,\xi_2)(X_t) = (z_1,z_2) \Big) = \nabla A (z_1,z_2).
\end{equation}
At equilibrium, the law of $X_t$ is, indeed, proportional to $\exp(-\beta (V - 
A^1_\infty \circ \xi_1 - A^2_\infty \circ \xi_2))$, which only differs from $\mu$ 
by a multiplicative function of $(\xi_1,\xi_2)$, which thus cancels out in the 
conditional average~\eqref{eq:cond_avg}.

More generally, to estimate canonical averages, one may {resort 
to} the classical unbiasing procedure:
\begin{equation}\label{eq:unbias}
\int_{\mathcal D} \varphi d\mu = \frac{\displaystyle\int_{\mathcal D} \varphi \, 
\exp(-\beta (A^1_\infty \circ \xi_1 + A^2_\infty \circ \xi_2)) \, \psi_\infty}
{\displaystyle\int_{\mathcal D} \exp(-\beta (A^1_\infty \circ \xi_1 + A^2_\infty
\circ \xi_2)) \, \psi_\infty},
\end{equation}
where $\psi_\infty \varpropto \exp(-\beta (V - A^1_\infty \circ \xi_1 - A^2_
\infty \circ \xi_2)) $ {stands for} the density of the equilibrium 
time marginal {law} of $X_t$, once $A^1_t$ and $A^2_t$ have 
reached a stationary state. 
{It is noteworthy} that in practice, it is always possible to 
{freeze} $A^1_t$ 
and $A^2_t$ at a 
given fixed time  $t_0$ and apply the unbiasing procedure above with the bias 
$A^1_{t_0} \circ \xi_1+A^2_{t_0} \circ \xi_2$ instead of $A^1_{\infty} \circ 
\xi_1+A^2_{\infty} \circ \xi_2$.

\subsection{Conformational equilibrium of the alanine dipeptide}

The first application of the method is a proof of concept making use of the
prototypical terminally blocked amino acid N--acetyl--N'--methylalanylamide
(NANMA), often referred to as alanine ``dipeptide''~\cite{Rossky1979}. The
molecular system consisted of NANMA immersed in a bath of 447 water molecules.
Conformational sampling was performed over a period of 10~ns with the 
{numerical scheme described here}, in which the~$\phi$ and~$\psi$ 
torsional angles of the backbone were
handled as independent order parameters covering each the [$-$180;$+$180] range
of the complete Ramachandran map~\cite{Ramachandran1963}. In the present
implementation of the method, the marginal biases, $A^1_\infty (\phi)$ and
$A^2_\infty(\psi)$, are periodical --- i.e. the average of their derivative is expected to
zero out. As a basis of comparison, a two--dimensional ABF calculation was
conducted over a period of 20~ns. For enhanced performances, the Ramachandran
map was split into four individual quadrants, corresponding to fully
independent simulations. Furthermore, a standard molecular-dynamics simulation
of equal length was performed, from which the~$\phi$ and~$\psi$ dihedral angles
were extracted to measure the exhaustiveness of the conformational sampling.

The simulations were carried out using the {\sc Namd} simulation
package~\cite{Phillips2005} in the isobaric--isothermal ensemble. The pressure
and the temperature were fixed at 1~bar and 300~K, respectively, employing the
Langevin piston algorithm~\cite{Feller1995} and softly damped Langevin
dynamics. The molecular system was replicated in the three directions of
Cartesian space by means of periodic boundary conditions. The particle--mesh
Ewald method~\cite{Darden1993} was employed to compute electrostatic
interactions. The $r$--RESPA multiple time--step
integrator~\cite{Tuckerman1992} was used with a time step of 2~fs and 4~fs for
short-- and long--range forces, respectively. Covalent bonds involving a
hydrogen atom were constrained to their equilibrium length. The short peptide
and its environment were described by the all-atom {\sc Charmm} force
field~\cite{MacKerell1998}.

\begin{figure}[!ht]
\includegraphics[width=\textwidth]{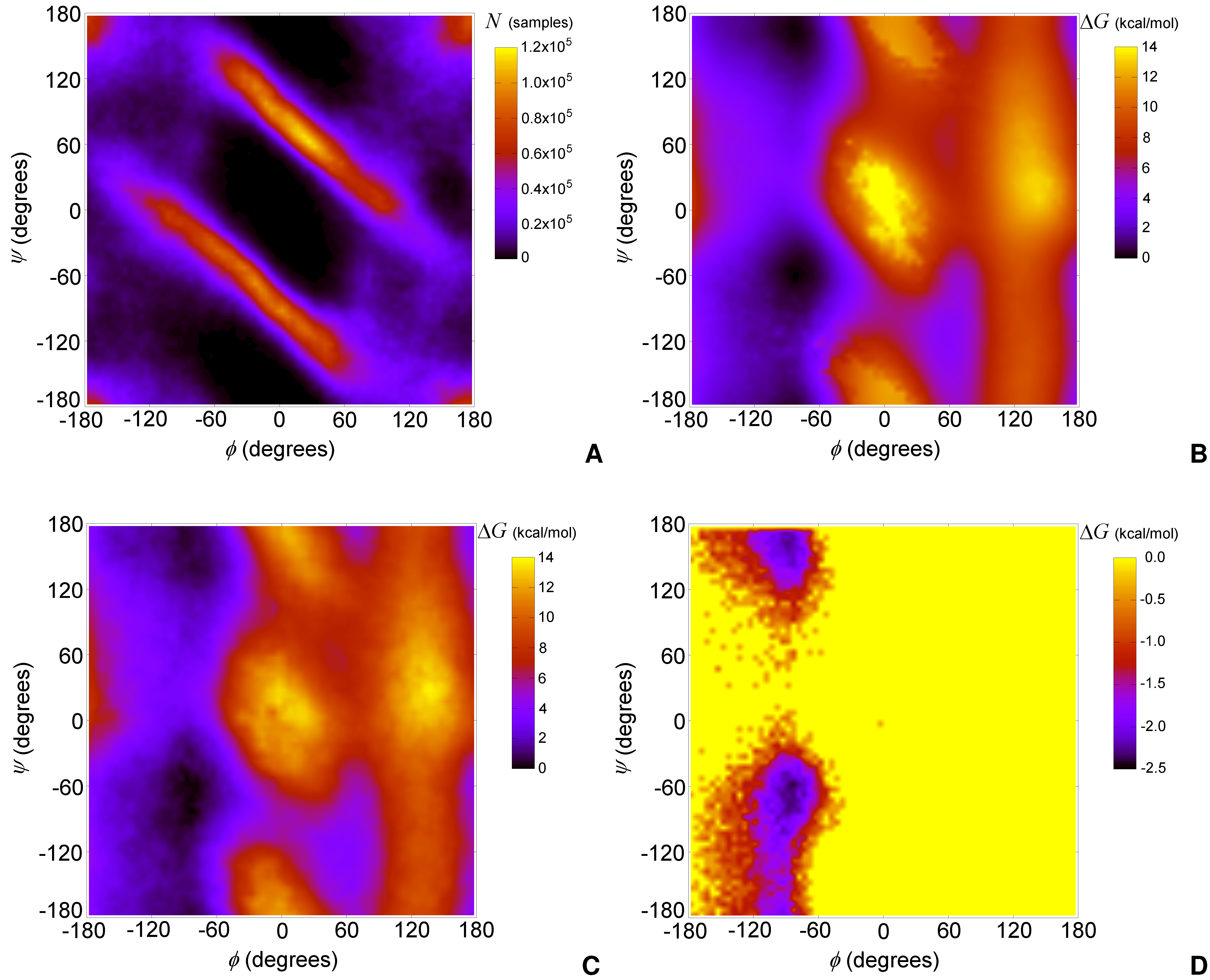}
\hspace*{0.8cm}
\begin{minipage}{13.1cm}
\caption{{\footnotesize
         Conformational equilibrium of the alanine dipeptide (NANMA) in an
         aqueous solution. Distribution of the ($\phi$, $\psi$) dihedral
         angles explored in the course of a 10--ns simulation, wherein the
         torsional angles of the backbone were handled as independent order 
         parameter~(A). Reconstruction of the conformational free-energy map,
         employing expression~(\ref{eq:recover})~(B). Conformational free-energy
         map obtained from a 20--ns, two-dimensional ABF calculation~(C). For
         comparison purposes, a similar map was generated from a 20--ns unbiased
         MD simulation~(D).
         \label{nanma}}}
\end{minipage}
\end{figure}

In the past thirty years, the conformational equilibrium of NANMA has been
investigated at different levels of detail, utilizing a variety of numerical 
schemes and
potential energy functions~\cite{Hagler1978,Rossky1979,Brady1985,Mezei1985,
Ravishanker1986,Anderson1988,Tobias1992,Pellegrini1996,Neria1996,Smart1997,
Chipot1998c,Smith1999,Bolhuis2000,Laio2002,Andricioaei2003,Chekmarev2004,
MacKerell2004,Wang2004,Jang2006,Branduardi2007,Neale2008,Feig2008}. Here, the 
investigation of NANMA is targeted at demonstrating
the ability of the method to recover within reasonable sampling time the
two-dimensional free-energy landscape that characterizes the conformational
equilibrium of the peptide in an aqueous environment. As can be observed in
Figure~\ref{nanma}, within 10~ns, the essential features of the ($\phi$, $\psi$)
conformational space appear to have been explored. Plotting the number of 
samples accrued in the course of the simulation brings to light the anticipated 
coupling between the backbone torsional angles. Concomitant variation of
the $\phi$ and $\psi$ angles is mirrored in the two parallel diagonals,
which appear to be oversampled. After convergence of the marginal 
biases, $A^1_\infty (\phi)$ and $A^2_\infty(\psi)$, the two-dimensional
free-energy landscape was reconstructed employing Equation~\eqref{eq:recover}.
Noteworthily, this map possesses the two expected pronounced minima corresponding
to a right-handed $\alpha$-helical conformation, often referred to as
$\alpha_R$, and to a $\beta$ strand, together
with ancillary local minima of higher free energy, associated to the so-called
$\alpha_D$ and $\alpha_L$ conformational states.

Integration over the basins delineating 
the $\alpha_R$ and the $\beta$ conformations 
yields a free-energy difference of about 0.2~kcal/mol,
in favor of the former state, congruent with 
previous computer simulations~\cite
{Tobias1992,Smart1997,Chipot1998c,MacKerell2004,Wang2004,Feig2008}.
Most importantly, the free-energy landscape inferred from the algorithm proposed 
herein is essentially identical to that obtained from a 20--ns ABF calculation
in ($\phi$, $\psi$) conformational space. A glimpse at Figure~\ref{nanma} is
sufficient to conclude that not only the two maps appear to be
almost interchangeable, but also, using the same bounds of integration over the
$\alpha_R$ and the $\beta$ basins, the free-energy differences agree 
quantitatively. It can, however, be contended that the present toy-example 
may be somewhat exaggeratedly simple to be representative of more challenging
instances, wherein the variables along which the local mean force is computed
are more strongly coupled. As has been discussed above, this scenario would
constitute a limiting case for the validity of the approach developed in this
contribution. Yet, as will be seen hereafter, estimators of the local mean 
force based on simulations handling unidimensional order parameters independently
can still be used profitably to describe accurately rugged multidimensional
free-energy landscapes wherein decoupling of the variables is
not straightforward. Last, to illustrate the role of importance sampling methods, the ($\phi$, 
$\psi$)-map regenerated from a 20--ns unbiased molecular-dynamics 
simulation. It is apparent from Figure~\ref{nanma} that only the lowest
free-energy states, i.e. $\alpha_R$ and $\beta$, have been visited, though
sampling is by and large too parsimonious to allow an acceptably precise
free-energy difference to be determined.

\subsection{Ion transport across a peptide nanotube}

In the second application of the method proposed herein, translocation of an 
halide ion through a chemically-tailored peptide nanotube is considered. Such
engineered synthetic channels arise from the self-assembly through an 
intermolecular hydrogen--bond network~\cite{Bong2001} of cyclic peptides
of alternated 
{\sc d}-- and {\sc l}--chirality~\cite{Ghadiri1993,Hartgerink1996}, in which all 
the side chains
are pointing outwards. Depending upon its amino-acid sequence, the resulting 
anti-parallel $\beta$--sheet--like hollow tubular structure can further
associate in the biological membrane with other channels to form nanopores, 
aggregate at the water--lipid interface prior to partitioning in the bilayer, 
disrupting in general the latter irreversibly, or simply span the membrane as
an independent entity~\cite{Ghadiri1994,Fernandez-Lopez2001}. Here, the peptide 
nanotube utilized consisted of eight
stacked cyclic peptides containing alternated
{\sc d}--leucine and {\sc l}--tryptophan residues and organized into
$cyclo$[\underline{L}W]$_4$ units~\cite{Dehez2007}. At thermodynamic equilibrium, the cyclic peptides are 4.7~\AA apart. The tubular structure was 
immersed in
a fully hydrated, thermalized palmitoyloleylphosphatidylcholine (POPC) 
bilayer formed by 48 lipid units in equilibrium with 1,572 water molecules.
The molecular assembly was replicated in the three directions of Cartesian
space. The initial dimensions of the simulation cell were 36 $\times$ 41 $\times$
79~\AA$^3$.

The free-energy landscape delineating permeation of a single chloride
ion through the synthetic channel was determined along the longitudinal,
$\zeta$, and radial, $\rho$, directions of the latter~\cite{Henin2010}. 
The surrogate, two-dimensional reaction coordinate $(\zeta,\rho)$ was constructed as a subset of 
cylindrical polar coordinates, namely the distance separating the halide
ion from the center of mass of the peptide nanotube projected onto its
longitudinal axis, associated with the distance between the ion from this
axis. Whereas the reaction pathway that connects the cytoplasm to the
periplasm would span approximately 40~\AA, the present investigation 
focuses on a 10--\AA\ segment, starting from the center of mass of the
open-ended tubular structure. This restrained pathway was discretized in
0.1--\AA\ wide bins, in which force samples were accrued. The 
molecular-dynamics simulations were performed in the same conditions
as described above for NANMA.

\begin{figure}[!ht]
\includegraphics[width=\textwidth]{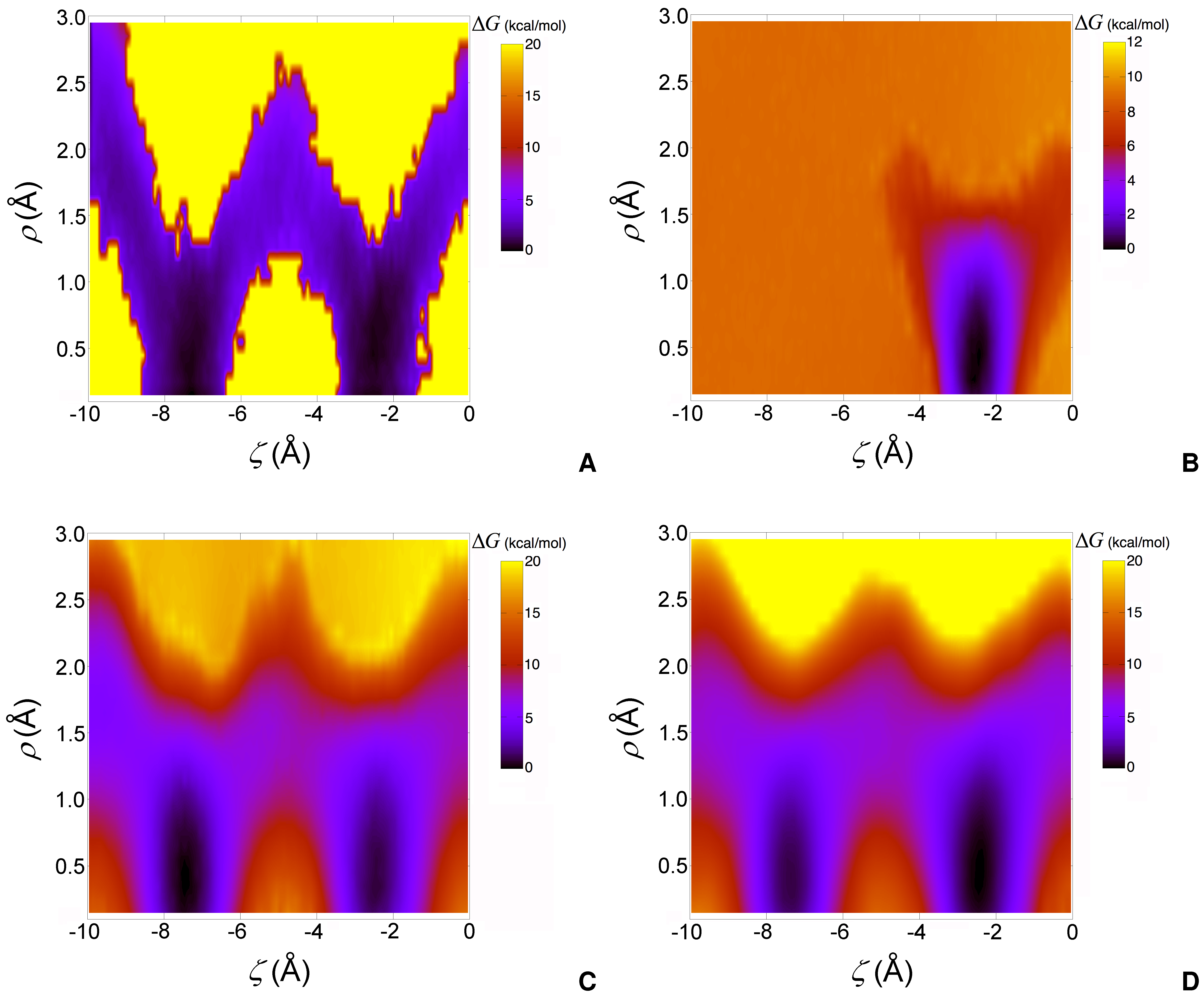}
\hspace*{0.8cm}
\begin{minipage}{13.1cm}
\caption{{\footnotesize
         Permeation of a chloride ion through a peptide 
         nanotube spanning a fully hydrated lipid bilayer. $\zeta$ denotes the 
         longitudinal axis of the synthetic channel and $\rho$, the
         radial direction. Reconstruction 
         employing expression~(\ref{eq:recover})
         of the free-energy landscape of ion permeation
         after 1~ns of sampling in which $\zeta$ and $\rho$ are treated
         as independent variables~(A). Free-energy map
         obtained from a 1--ns two-dimensional ABF calculation~(B).
         10--ns ABF calculation, using as a starting point the 
         two-dimensional gradients recovered from the 1--ns
         simulation in which $\zeta$ and $\rho$ are handled
         independently and
         employing expression~\eqref{eq:grad_A}~(C).
         As a basis of comparison, ($\zeta$, $\rho$) free-energy
         landscape inferred from a converged 
         30--ns two-dimensional ABF calculation~(D). The lowest free-energy regions found at $\zeta = -2.4$~\AA\ and $\zeta =-7.1$~\AA\
correspond to an in-plane chelation of the halide ion, where the
latter is located at the geometric center of the cyclic peptide.
         \label{cyclo1}}}
\end{minipage}
\end{figure}

Chemically--engineered peptide nanotubes possess the ability of conducting
ions~\cite{Sanchez-Quesada2002}. 
Assuming an appropriate amino-acid sequence, individual tubular
structures can insert in the lipid bilayer, where they act as transmembrane
channels~\cite{Kim1998}. 
That such channels can be permeated by a small ions has been
investigated at the theoretical level, employing a variety of 
methods~\cite{Asthagiri2002,Tarek2003,Hwang2006,Hwang2006a}.
It has been suggested recently that whereas diffusion of a sodium ion
through a synthetic channel formed of eight $cyclo$[\underline{L}W]$_4$
units and immersed in a fully hydrated POPC bilayer is essentially 
unhampered~\cite{Dehez2007}, 
the same cannot be said for a chloride ion shuttled across
the cavity of an identical peptide nanotube~\cite{Henin2010}. 
This result can be rationalized
to a large extent by the lesser hydration number of the cation ---
viz. approximately 4--6, compared
to that of the halide ion --- viz. approximately
6--8~\cite{Asthagiri2002}, 
which must undergo considerable dehydration to
enter the open-ended tubular structure. In the midst of the latter, however,
the free-energy landscape characterizing ion permeation is roughly similar
for both species, in--plane coordination being enthalpically favored because
it allows the ion to be better hydrated~\cite{Dehez2007}.

As can be seen in Figure~\ref{cyclo1},
this preference is reflected in the free-energy landscape obtained from the
method proposed herein, handling $\zeta$ and $\rho$ as independent variables.
Within 1~ns of sampling, the entire reaction pathway is explored, 
following what appears to be a minimum--action path. It is noteworthy 
that the chloride ion does not diffuse along a rectilinear path,
collinear to the longitudinal axis of the synthetic channel, e.g.
at $\zeta$ = 0~\AA, but rather avoids the large free-energy barrier
of mid--plane coordination by grazing the wall of the tubular structure
to interact with the amino groups of the $cyclo$[\underline{L}W]$_4$
units. Repeated simulations,
using different initial momenta, yield comparable landscapes, yet
wherein the higher free-energy regions are essentially never visited.
Interestingly enough, the free-energy minima found at about $-$2.4 and 
$-$7.1~\AA\ have virtually the same depth.
It is remarkable that a limited simulation length of 1~ns would provide a
consistent picture of the path followed by the chloride ion over the 
10--\AA\ stretch of the peptide nanotube, when repeated two-dimensional 
ABF calculations of equal length only sample a small fraction of 
the ($\zeta$, $\rho$) configurational space, as shown in Figure~\ref{cyclo1}.

\begin{figure}[!ht]
\includegraphics[width=\textwidth]{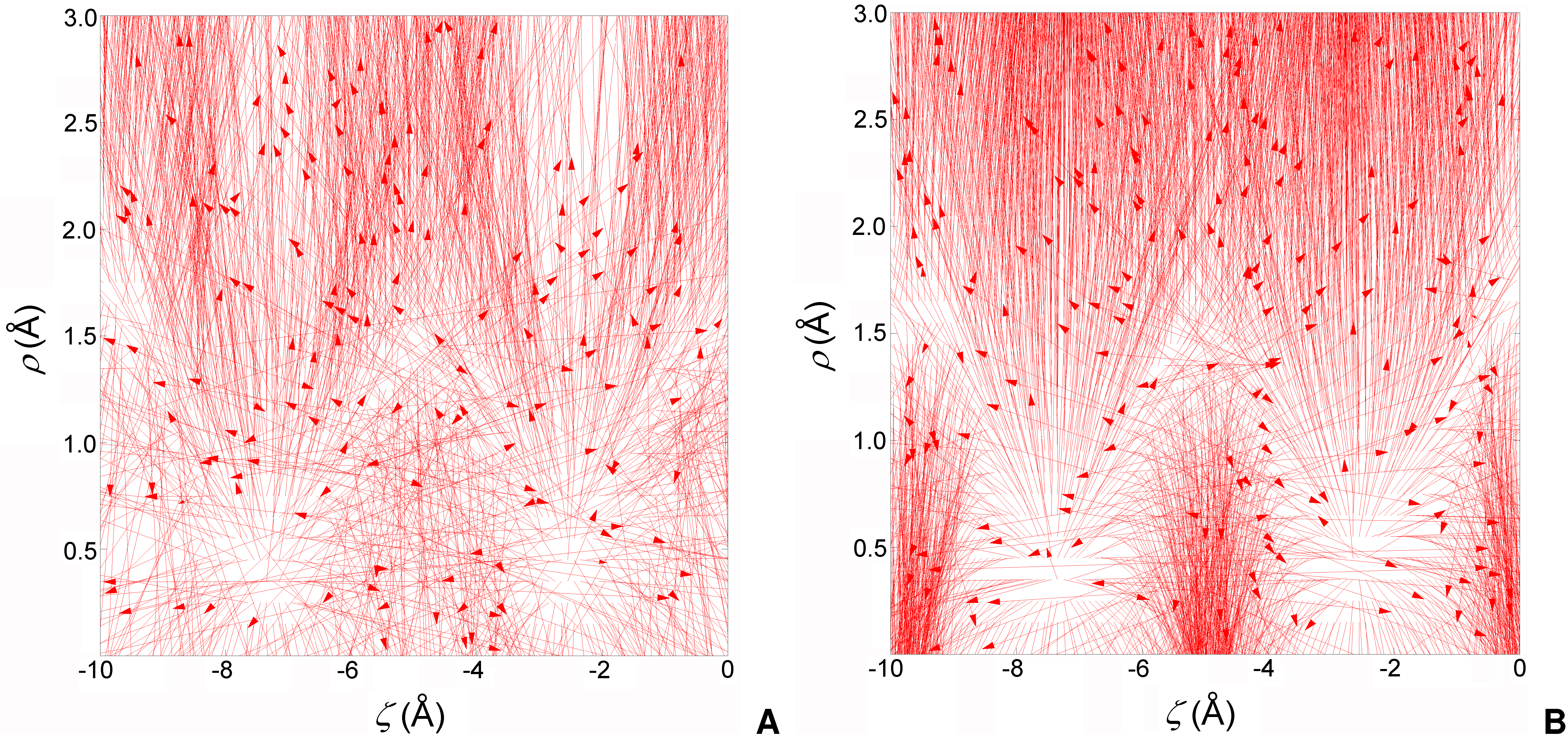}
\hspace*{0.8cm}
\begin{minipage}{13.1cm}
\caption{{\footnotesize
         Permeation of a chloride ion through a peptide 
         nanotube spanning a fully hydrated lipid bilayer. $\zeta$ denotes the 
         longitudinal axis of the synthetic channel and $\rho$, the
         radial direction. Gradient of the free energy inferred from
         a 1--ns simulation wherein 
         $\zeta$ and $\rho$ are handled independently, employing
         expression~(\ref{eq:grad_A})~(A). For comparison purposes,
         two-dimensional gradient of the free energy obtained from
         a 30--ns ABF calculation in ($\zeta$, $\rho$) configurational
         space~(B).
         \label{cyclo2}}}
\end{minipage}         
\end{figure}

Unfortunately, while the ABF calculation progressively explores the free-energy
landscape of ion permeation, reaching convergence within 30~ns, a simulation of 
equal length based on the numerical scheme introduced here only marginally 
improves the picture drawn from the aforementioned short 1--ns run.  
This glaring shortcoming of the method can be ascribed in large measure to 
antagonist effects of the biases in the $\zeta$ and in the $\rho$ directions,
deteriorating mutually any progress made by the two variables.
It does not mean, however, that the valuable, yet incomplete description of 
the reaction pathway cannot be utilized profitably to recover the correct
free-energy landscape, possibly faster than a classical ABF calculation would. 

As has been discussed previously, Equation~\eqref{eq:grad_A} allows the
gradient of the free energy to be estimated on the basis of independent
measures of the force acting along the two order parameters, $\zeta$ and $\rho$.
In Figure~\ref{cyclo2}, the reference two-dimensional gradient of the free
energy inferred from a 30--ns ABF calculation is compared to an approximation
thereof obtained after 1~ns of sampling. Although the two vector fields show
significant discrepancies, they also retain common characteristic features,
in particular in the region of lower free energy, visited appropriately by
the algorithm described herein. In turn, the approximate gradient can be
employed as a starting point for a separate ABF calculation in ($\zeta$, $\rho$) 
configurational space, with the hope that the initial guess of a minimum--action
pathway might boost exploration of the complete free-energy landscape.
As highlighted in Figure~\ref{cyclo1}, this appears to be, indeed, true ---
after 10~ns, the map inferred from the separate two-dimensional ABF
run possesses a topology essentially identical to that of the reference,
30--ns simulation. However not interchangeable, the two free-energy 
landscapes agree quantitatively in the low free-energy regions ---
i.e 0 $\leq \rho \leq$ 2~\AA, albeit only qualitatively so in the
higher free-energy regions corresponding to the wall of the open-ended
tubular structure.

\bigskip

{\bf Acknowledgments.}\\
The authors gratefully acknowledge Jérôme Hénin for stimulating
discussions. They are indebted to the Grand Équipement National de Calcul
Intensif and the Centre Informatique National de l'Enseignement
Supérieur for provision of computer time.
This work is funded by the Agence Nationale de la
Recherche (MEGAS project, grant ANR-09-BLAN-0216-01) and the Institut National de la Recherche en Informatique et en Automatique (ARC Hybrid project). 


\end{document}